\documentclass[11pt,reqno]{amsart}
\usepackage[numbers, square]{natbib}
\usepackage{mathptmx}
\usepackage{amsmath}
\usepackage{amscd}
\usepackage{amssymb}
\usepackage{amsthm}
\usepackage{enumerate}
\usepackage{xspace}
\usepackage[all,tips]{xy}
\usepackage[dvips]{graphicx}
\usepackage{verbatim}
\usepackage{syntonly}
\usepackage{hyperref}
\usepackage{amsmath, amsthm, graphics, amssymb,fullpage,color, epsfig,url}
\usepackage{indentfirst}
\usepackage{color}

\theoremstyle{plain}
\newtheorem{thm}{Theorem}
\newtheorem{lem}{Lemma}

\theoremstyle{definition}

\newenvironment{pf}
{\begin{proof}} {\end{proof}}

\newcommand{\disp}{\displaystyle}

\DeclareMathOperator{\bigO}{\Cal{O}}

\newcommand{\eps}{\varepsilon}
\newcommand{\vp}{\varphi}

\newcommand{\be}{\beta}
\newcommand{\ga}{\gamma}

\newcommand{\De}{\Delta}

\newcommand{\la}{\lambda}
\newcommand{\La}{\Lambda}

\newcommand{\nid}{\noindent}

\newcommand{\iny}{\infty}

\newcommand{\LP}{\Delta}
\newcommand{\gr}{\nabla}
\newcommand{\pri}{\prime}

\newcommand{\norm}[1]{\left\vert \left\vert #1\right\vert\right\vert}

\newcommand{\abs}[1]{\left\vert#1\right\vert}
\newcommand{\set}[1]{\left\{#1\right\}}
\newcommand{\brac}[1]{\left[#1\right]}
\newcommand{\pr}[1]{\left( #1 \right) }
\newcommand{\ang}[1]{\langle #1 \rangle }

\newcommand{\Cal}[1]{\ensuremath{\mathcal{#1}}}

\newcommand{\N}{\ensuremath{\mathbb{N}}}

\newcommand{\R}{\ensuremath{\mathbb{R}}}
\newcommand{\Z}{\ensuremath{\mathbb{Z}}}
\newcommand{\C}{\ensuremath{\mathbb{C}}}

\numberwithin{equation}{section}
\numberwithin{lem}{section}
\numberwithin{cor}{section}

\date{}
\begin{document}
\author{Blair Davey}
\address{Department of Mathematics, University of Minnesota}
\email{edavey@umn.edu}

\title{A Meshkov-type construction for the borderline case}
\maketitle

\begin{abstract}
We construct functions $u: \R^2 \to \C$ that satisfy an elliptic eigenvalue equation of the form  $-\LP u + W \cdot \gr u + V u = \la u$, where $\la \in \C$, and $V$ and $W$ satisfy $\abs{V\pr{x}} \lesssim \ang{x}^{-N}$, and $\abs{W\pr{x}} \lesssim \ang{x}^{-P}$, with $\min\set{N, P} = 1/2$.  For $\abs{x}$ sufficiently large, these solutions satisfy $\abs{u(x)} \lesssim \exp\pr{- c\abs{x}}$.  In the author's previous work, examples of solutions over $\R^2$ were constructed for all $N, P$ such that $\min\set{N,P} \in [0, 1/2)$. These solutions were shown to have the optimal rate of decay at infinity. A recent result of Lin and Wang shows that the constructions presented in this note for the borderline case of $\min\set{N, P} = 1/2$ also have the optimal rate of decay at infinity. \\

\end{abstract}

\section{Introduction}

In this note, we follow up on previous work and address the missing cases of constructions of solutions in $\R^2$ that satisfy an elliptic eigenvalue equation and have the optimal rate of decay at infinity.  Previously, sharp constructions of solutions corresponding to values of $\be_c > 1$ and $\be_c < 1$ were presented.  (The definition of $\be_c$ is given below.)  These sharp constructions were presented in \cite{D13} in conjunction with quantitative unique continuation estimates for eigenfuctions of the magnetic Schr\"odinger operator.  For $\be_c > 1$, the optimal constructions mimic those of Meshkov \cite{M89} and hold only in dimension $2$.  For $\be_c < 1$, sharp radial constructions exist in any dimension greater than 1.  We will show that a modification of the Meshkov-type constructions from \cite{D13} (that only worked for $\be_c > 1$) gives rise to a set of constructions at the borderline case of $\be_c = 1$.  A modification to the definitions of certain cutoff functions leads to a simplified estimate, which in turn allows us to push a construction that only worked for $\be_c >1$ down to $\be_c =1$.  

The main theorem of Lin and Wang in \cite{LW13} is a generalization of the quantitative estimates from \cite{D13}, and it holds for any value of $\be_c$.  In particular, the theorem from \cite{LW13} holds when $\be_c = 1$, the case that is missing from \cite{D13}.  These new constructions are especially interesting because they prove that the result of Lin and Wang is sharp when $\be_c = 1$.

To understand the importance of the new constructions, we will first summarize the main results from \cite{D13}.  Recall that $\langle x \rangle = \sqrt{1 + \abs{x}^2}$.  Let $\la \in \C$ and suppose that $u$ is a solution to 
\begin{equation}
- \LP u + W\cdot \gr u + Vu = \la u  \;\; \textrm{in} \;\; \R^n,
\label{epde}
\end{equation} 
where 
\begin{equation}
|V(x)| \le A_1{\langle x\rangle^{-N}}, 
\label{vBd}
\end{equation}
\begin{equation}
|W(x)| \le A_2{\langle x\rangle^{-P}},
\label{wBd}
\end{equation}
for $N, P, A_1, A_2 \ge 0$.  Assume also that $u$ is bounded,
\begin{equation}
\norm{u}_{\iny} \le C_0,
\label{uBd}
\end{equation}
and normalized,
\begin{equation}
u(0) \ge 1.
\label{L2lBd}
\end{equation}
Define $\be_c = \disp \max\set{2 - 2P, \frac{4-2N}{3}}$, $\be_0 = \max\set{\be_c, 1}$.  For large $R$, let
$$\mathbf{M}(R) = \inf_{|x_0| = R}\norm{u}_{L^2\pr{B_1(x_0)}}.$$

The following theorem is the main result of \cite{D13}.

\begin{thm}
Assume  that the conditions described above in (\ref{epde})-(\ref{L2lBd}) hold.  Then there exist constants $\tilde C_5(n)$, $C_6\pr{n, N, P}$, $C_7\pr{n, N, P, A_1, A_2}$, $R_0\pr{n, N, P, \la, A_1, A_2, C_0}$, such that for all $R \ge R_0$,
\begin{enumerate}[(a)]
\item if $\be_c > 1$ ($\be_0 = \be_c$), then
\begin{equation}
\mathbf{M}(R) \ge \tilde C_5\exp\pr{-C_7 R^{\be_0}(\log R)^{C_6}},
\label{Mesta}
\end{equation}
\label{partA}
\item if $\be_c < 1$ ($\be_0 = 1$), then
\begin{equation}
\mathbf{M}(R) \ge \tilde C_5\exp\pr{-C_7 R(\log R)^{C_6\log\log R}}.
\label{Mestb}
\end{equation}
\label{partB}
\end{enumerate}  
\label{MEst}
\end{thm}

Notice that Theorem \ref{MEst} does not address the case of $\be_c = 1$.  In \cite{LW13}, the authors applied the methods used to prove Theorem \ref{MEst} and established a more general version of that theorem. In doing so, they proved the appropriate estimate for the missing case of $\be_c = 1$.  The following is a statement of a specific case of Theorem 1.1 from \cite{LW13}:

\begin{thm}
Assume  that the conditions described above in (\ref{epde})-(\ref{L2lBd}) hold.  Then there exist constants $\tilde C_5(n)$, $C_6\pr{n, N, P}$, $C_7\pr{n, N, P, A_1, A_2}$, $R_0\pr{n, N, P, \la, A_1, A_2, C_0}$, such that for all $R \ge R_0$,
 if $\be_c = 1$, then
\begin{equation}
\mathbf{M}(R) \ge \tilde C_5\exp\pr{-C_7 R (\log R)^{C_6 \ga\pr{R}}},
\end{equation}
where $\disp \ga\pr{R} = \frac{ \pr{\log R} \pr{\log\log\log R}}{\pr{\log \log R}^2}$.
\label{LWC}
\end{thm}

By examining Theorem \ref{MEst}, we see that, up to logarithmic factors, Theorem \ref{LWC} is precisely the estimate that one would expect to be true for the case of $\be_c = 1$.  

The following theorem from \cite{D13} shows that, under certain conditions, there are constructions that prove that Theorem \ref{MEst} is sharp (up to logarithmic factors).  

\begin{thm}
For any $\la \in \C$, $N, P \ge 0$ chosen so that either 
\begin{enumerate}[(a)]
\item $\disp \be_0 = \be_c > 1$ and $n = 2$ or 
\label{consa}
\item $\be_c < 1$ and $\la \notin \R_{\ge 0}$,
\label{consb}
\end{enumerate}
there exist complex-valued potentials $V$ and $W$ (exactly one of which is equal to zero) and a non-zero solution $u$ to (\ref{epde}) such that 
\begin{equation}
|V(x)| \le C{\langle x\rangle^{-N}},
\label{vBd2}
\end{equation}
\begin{equation}
|W(x)| \le C{\langle x\rangle^{-P}}.
\label{wBd2}
\end{equation}
Furthermore, $$|u(x)| \le C\exp\pr{-c|x|^{\be_0}\pr{\log |x|}^A},$$
for some constant $A \in \set{-1, 0}$.
\label{cons}
\end{thm}

The next theorem, the main result of this note, shows that Theorem \ref{LWC} is sharp and therefore fills in all of the remaining gaps for $n =2$.  We will adapt the methods from \cite{D13}, based on those from \cite{M89}, to prove the following.

\begin{thm}
For any $\la \in \C$, $N, P \ge 0$ chosen so that  $\disp \be_0 = \be_c=1$ and $n = 2$, we have the following:
\begin{enumerate}[(a)]
\item If $\disp \be_0 = \frac{4-2N}{3} = 1$, then there exists a potential $V$ and an eigenfunction $u$ such that
\begin{equation}
\LP u + \la u = V u,
\label{VPDE}
\end{equation}
where 
\begin{align}
\abs{V\pr{x}} &\le C \log \abs{x} \ang{x}^{-1/2}
\label{Vbd} 
\end{align}
\item  If $\be_0 = 2-2P = 1$, then there exists a potential $W$ and an eigenfunction $u$ such that
\begin{equation}\LP u + \la u = W \cdot \gr u,
\label{WPDE}
\end{equation}
where
\begin{align}
\abs{W\pr{x}} &\le C \ang{x}^{-1/2}
\label{Wbd}
\end{align}
\item If $\disp \be_0 = \frac{4-2N}{3} = 1$ and $\la > 0$, then there exists a potential $V$ and an eigenfunction $u$ such that (\ref{VPDE}) holds with
\begin{align}
\abs{V\pr{x}} &\le C \ang{x}^{-1/2}.
\label{VXbd} 
\end{align}
\end{enumerate}
In all cases, 
\begin{equation}
|u(x)| \le C\exp\pr{-c|x|}.
\label{uBd}
\end{equation}
\label{consBe}
\end{thm}

This article is organized as follows.  In \S \ref{ProofDesc}, the general approach to the proof of Theorem \ref{consBe} is described.  In particular, the statements of all necessary lemmas are presented and their application is indicated.  \S \ref{prop} assumes Lemmas \ref{meshN}, \ref{meshP} and \ref{meshNX} to prove Theorem \ref{consBe}(a), (b) and (c), respectively.  Each lemma is then proved in a separate section.  In \S \ref{lemN}, the proof of Lemma \ref{meshN} is presented.  This lemma gives constructions on an annulus for the case of $\be_c = \frac{4-2N}{3} = 1$.  The next section, \S \ref{lemP}, proves the corresponding lemma for $\be_c = 2 - 2P = 1$.  And in \S \ref{lemNX}, we prove a lemma that shows that under additional assumptions on the eigenvalue, $\la$, we may remove the logarithmic term from estimate (\ref{Vbd}).  

%
%
\section{The description of the proof of Theorem \ref{consBe}}
\label{ProofDesc}

Theorem \ref{consBe} is proved with a Meshkov-type construction.   To give the constructions for Theorem \ref{consBe}, we first construct solutions on annular regions.  The annular constructions are described in the lemmas below.  Once the lemmas have been established, the proof of Theorem \ref{consBe} consists of showing that the solutions on annuli can be put together to give solutions over all of $\R^2$ with the appropriate decay properties.  

For $\la \in \C$, use the principal branch to define
$$ \mu_n( r ) := \exp\pr{n\brac{\log\pr{\sqrt{1 - \frac{\la r^2}{n^2}}+1}- \sqrt{1 - \frac{\la r^2}{n^2}} - \log 2 + 1}}.$$
As we will specify below, $n \sim 2\sqrt{\La} r$, so $ \abs{\frac{\la r^2}{n^2}} < 1$. It follows that $ \Re \pr{1 - \frac{\la r^2}{n^2}} > 0$, so all square root terms are well defined (and have positive real part) with this choice of branch cut.  Since the argument for the logarithmic term has real part greater than 1, that term, and hence the function $\mu_n$, is well defined with this branch choice.
A power series expansion of the exponent gives
$$\mu_n\pr{r} = \exp\pr{\frac{\la r^2}{4n} + \frac{\la^2 r^4}{32 n^3} + \frac{\la^3 r^6}{96 n^5} + \ldots}.$$
Whenever $ \abs{\frac{\la r^2}{n^2}} < 1$, the power series in the exponent converges everywhere.

For the case described in Lemma \ref{meshNX} below, $ 1 -\frac{\la r^2}{n^2} > 0$, so again, all terms are well-defined.

The following lemma leads to the proof of Theorem \ref{consBe}(a).

\begin{lem}
Suppose $\rho$ is a large positive number $(\rho > \rho_0 > 0)$, $\disp \be_0 = \frac{4-2N}{3} = 1$, $\La = \max\set{1, \abs{\la}}$, $n \in \N$ is such that $\disp \abs{n -2\sqrt{\La} \rho} \le 1$ and $k \in \N$ is such that $\disp \abs{k - 12\sqrt{\La}\sqrt{\rho}} \le 1$.  
Then in the annulus $\brac{\rho, \rho + 6\sqrt{\rho}}$ it is possible to construct an equation of the form (\ref{VPDE}) and a solution $u$ of this equation such that the following hold:
\begin{enumerate}
\item (\ref{Vbd}), where $C$ does not depend on $\rho$, $n$ or $k$.
\item If $r \in \brac{\rho, \rho + 0.1\sqrt{\rho}}$, then $u = r^{-n}e^{-in\vp}\mu_n$. \\
If $r \in \brac{\rho + 5.9\sqrt{\rho}, \rho + 6\sqrt{\rho}}$, then $u = ar^{-(n+k)}e^{-i(n+k)\vp}\mu_{n+k}$, for some $a\in \C\setminus\set{0}$.
\label{p2}
\item Let $m(r) = \max\set{|u(r, \vp)| : 0 \le \vp \le 2\pi}$.  Then there exists a $c > 0$, not depending on $\rho$, $n$ or $k$, such that
\begin{equation}
\ln m(r) - \ln m(\rho) \le - c\int_\rho^r dt
\label{Nint}
\end{equation}
for any $r \in \brac{\rho, \rho + 6\sqrt{\rho}}$.
\label{p3}
\end{enumerate}
\label{meshN}
\end{lem}

This next lemma is used to give Theorem \ref{consBe}(b).

\begin{lem}
Suppose $\rho$ is a large positive number $(\rho > \rho_0 > 0)$, $\be_0 = 2-2P = 1$, $\La = \max\set{1, \abs{\la}}$, $n \in \N$ is such that $\disp \abs{n -2\sqrt{\La} \rho} \le 1$ and $k \in \N$ is such that $\disp \abs{k - 12\sqrt{\La}\sqrt{\rho}} \le 1$.  
Then in the annulus $\brac{\rho, \rho + 6\sqrt{\rho}}$ it is possible to construct an equation of the form (\ref{VPDE}) and a solution $u$ of this equation such that the following hold:
\begin{enumerate}
\item (\ref{Wbd}), where $C$ does not depend on $\rho$, $n$ or $k$.
\item from Lemma \ref{meshN}
\item from Lemma \ref{meshN}
\end{enumerate}
\label{meshP}
\end{lem}

The following variation of Lemma \ref{meshN} allows us to remove the logarithmic term from estimate (\ref{Vbd}).

\begin{lem}
Suppose $\rho$ is a large positive number $(\rho > \rho_0 > 0)$, $\disp \be_0 = \frac{4-2N}{3} = 1$, $\la > 0$, $n \in \N$ is such that $\disp \abs{n -\sqrt{{\la}}\pr{ \rho + 8 \sqrt{\rho}}} \le 1$ and $k \in \N$ is such that $\disp \abs{k - 12\sqrt{{\la}}\sqrt{\rho}} \le 1$.  
Then in the annulus $\brac{\rho, \rho + 6\sqrt{\rho}}$ it is possible to construct an equation of the form (\ref{VPDE}) and a solution $u$ of this equation such that the following hold:
\begin{enumerate}
\item (\ref{VXbd}), where $C$ does not depend on $\rho$, $n$ or $k$.
\item from Lemma \ref{meshN}
\item from Lemma \ref{meshN}
\end{enumerate}
\label{meshNX}
\end{lem}

\section{The proof of Theorem \ref{consBe}}
\label{prop}

We now use the lemmas to construct examples that prove Theorem \ref{consBe}.  

\begin{pf}[Proof of Theorem \ref{consBe}]
We recursively define a sequence of numbers $\set{\rho_j}_{j=1}^\iny$.  For $\rho_1$, we choose a sufficiently large positive number.  Then if $\rho_j$ has been chosen, we set $\rho_{j+1} = \rho_j + 6\sqrt{\rho_j}$.  Suppose that $N$ and $P$ are chosen so that $\be_0 = 1$.  In order to use Lemmas \ref{meshN} and \ref{meshP}, we let $n_j = \Big\lfloor 2\sqrt{\La}\rho_j \Big\rfloor = 2\sqrt{\La}\rho_j - \eps_j$ and $k_j = n_{j+1} - n_j$.  We must estimate $k_j$:
\begin{align*}
k_j &= 2\sqrt{\La}\rho_{j+1} - \eps_{j+1} - 2\sqrt{\La}\rho_j + \eps_j \\
&=  2\sqrt{\La}\pr{\rho_{j+1}  - \rho_j}  - \De_{\eps,j} \\
&=  12\sqrt{\La}\sqrt{\rho_j}  - \De_{\eps,j} 
\end{align*}
Therefore, $ \abs{k_j - 12\sqrt{\La}\sqrt{\rho_j}} \le 1$. For Lemma \ref{meshNX}, we set $n_j = \Big\lfloor \sqrt{{\la}}\pr{ \rho_j + 8 \sqrt{\rho_j}} \Big\rfloor$, $k_j = n_{j+1} - n_j$, and establish the estimate for $k_j$ in a similar way, assuming that $\rho_j$ is sufficiently large.  \\

For $j = 1, 2, \ldots$, we let $u_j$ denote the solutions of equations of the form (\ref{VPDE}) or (\ref{WPDE}), denoted by $L_j u_j = 0$.  By Lemmas \ref{meshN}-\ref{meshNX}, these equations and their solutions are constructed in the annulus $\set{\rho_j \le r \le \rho_{j+1}}$.  The required decay estimate for the potentials is given by (1) from each lemma.  Result (2) from each lemma shows that  $\disp u_j(\rho_j, \vp) = \rho_j^{-n_j}e^{-in_j\vp}\mu_{n_j}\pr{\rho_j}$ and $\disp u_j(\rho_{j+1}, \vp) = a_j\rho_{j+1}^{-n_{j+1}}e^{-in_{j+1}\vp}\mu_{n_{j+1}}\pr{\rho_{j+1}}$.

Set $\rho_0 = 0$ and denote by $g_0(r)$ a smooth function in $\brac{0, \rho_1}$ such that $g_0(r) = r^{n_1}$ in a neighbourhood of $0$ while $g_0(r) = r^{-n_1}$ in a neighbourhood of the point $\rho_1$.  We suppose also that $g_0(r) > 0$ on $(0, \rho_1)$.  Let $u_0 = g_0(r)e^{-in_1\vp}\mu_{n_1}( r )$ and denote by $L_0u_0 = 0$ the equation of the form (\ref{VPDE}) or (\ref{WPDE}) which the function $u_0$ satisfies.

We define a differential operator $L$ in $\R^2$ by setting $L = L_j$ for $\rho_j \le r \le \rho_{j+1}$, $j = 0, 1, \ldots$.  We define a $C^2$ function $u$ on $\R^2$ by setting $u(r, \vp) = u_j(r, \vp)$ if $\rho_j \le r \le \rho_{j+1}$, $j = 0, 1$, and
$$u(r, \vp) = \pr{\prod_{i=1}^{j-1}a_i}u_j(r, \vp),$$
if $\rho_j \le r \le \rho_{j+1}$, $j = 2, 3,\ldots$.  Then it is clear that $u$ satisfies $Lu = 0$ in $\R^2$.   

We must now estimate $\abs{u}$.  Set $m(r) = \max\set{|u(r, \vp)| : 0 \le \vp \le 2\pi}$.  For a given $r \in \R^+$, we choose $\ell \in \Z$ so that $\rho_\ell \le r \le \rho_{\ell+1}$.  Then
\begin{align*}
\ln m(r) = &(\ln m(r) - \ln m(\rho_\ell)) + (\ln m(\rho_\ell) - \ln m(\rho_{\ell-1})) + \cdots + (\ln m(\rho_2) - \ln m(\rho_{1})) + \ln m(\rho_1)
\end{align*}
By (3) from Lemmas \ref{meshN}-\ref{meshNX}, for sufficiently large $r$, we have
$$\ln m(r) \le - c\int_{\rho_1}^r  dt + \ln m(\rho_1),$$
so that
$$m(r) \le C\exp\pr{-c r}$$
and (\ref{uBd}) holds.
\end{pf}

\section{Proof of Lemma \ref{meshN}}
\label{lemN}

Since Lemmas \ref{meshN}-\ref{meshNX} are so similar, it is not surprising that their proofs are as well.  We will present the more complicated proof first, that of Lemma \ref{meshN}.  We will then describe the proofs of Lemmas \ref{meshP} and \ref{meshNX} in subsequent sections.  

The following constructions are very similar to those presented in \cite{D13}.  The significant difference between these examples is the more careful choice of cutoff functions.  By choosing slightly more complicated cutoff functions in the current construction, we are able to reduce the bound on the modulus of the functions.  This allows us to eliminate the $\ln 2$ term that previously appeared in the estimates given in (3) of each lemma.  In turn, we are now able to sum these estimates to a negative number when $\be_c = 1$.  

\begin{pf}[Proof of Lemma \ref{meshN}]
As $r$ increases from $\rho$ to $\rho + 6\sqrt{\rho}$, we rearrange equation (\ref{VPDE}) and its solution $u$ so that all of the above conditions are met.  This process is broken down into four major steps.

Throughout this proof, the number $C$ is a constant that is independent of $\rho$, $n$ and $k$. \\
\nid\emph{Step 1: $r \in \brac{\rho, \rho + 2\sqrt{\rho}}$}.  During this step, the function $u_1 = r^{-n}e^{-in\vp}\mu_{n}( r )$ is rearranged to to a function of the form
$u_2 = -br^{-n+2k}e^{iF\pr{\vp}}\mu_{n-2k}\pr{r}$, both of which satisfy an equation of the form (\ref{VPDE}), where $b$ is a complex number and $F$ is a function that will be defined shortly.

Let $\vp_m = \frac{2 \pi m}{2n + 2k}$, for $m = 0, 1 , \ldots, 2n + 2k -1$.  Then $\disp \set{\vp_m}_{m=0}^{2n+2k-1}$ is the set of all solutions to $\disp e^{-in\vp} - e^{i(n+2k)\vp} = 0$ on $\set{0 \le \vp \le 2\pi}$.  Let $T = \frac{\pi}{n+k}$.  On $\brac{0, T}$, we define $f$ to be a $C^1$ function such that $f(\vp) = -4k$ for $\vp \in \brac{0, T/5} \cup \brac{4T/5, T}$.  We also require that $f$ satisfies the following:
\begin{align}
&-4k \le f(\vp) \le 5k, \;\; 0 \le \vp \le T,
\label{f1} \\
&\int_0^T f(\vp) d\vp = 0,
\label{f2} \\
&|f^\pri(\vp)| \le \tfrac{Ck}{T} = \tfrac{Ck(k+n)}{\pi},  \;\; 0 \le \vp \le T.
\label{f3}
\end{align}
We extend $f$ periodically (with period $T$) to all of $\R$ and set
$$\Phi(\vp) = \int_0^\vp f(t) dt.$$
By (\ref{f2}), $\Phi$ is $T$-periodic and $\Phi(\vp_m) = \Phi(mT) = 0$.  Furthermore, $\Phi$ is $2\pi$-periodic.  By (\ref{f1})-(\ref{f3}), the following facts hold for all $\vp \in \R$:
\begin{align}
&|\Phi(\vp)| \le 5kT = \tfrac{5\pi k}{n+k},
\label{Phi1} \\
&|\Phi^\pri(\vp)| \le 5k,
\label{Phi2} \\
&|\Phi^{\pri\pri}(\vp)| \le Ckn.
\label{Phi3} 
\end{align}
Also, for all $\vp \in \set{|\vp - \vp_m| \le T/5}$, 
\begin{equation}
\Phi(\vp) = -4k(\vp - \vp_m) = -4k\vp + b_m,
\label{Phi4}
\end{equation}
where $b_m$ is some real number.

Set 
\begin{equation}
F(\vp) = (n+2k)\vp + \Phi(\vp).
\label{FDef}
\end{equation}  
If $|\vp - \vp_m| \le T/5$, then $u_2 = -be^{ib_m}r^{-(n-2k)}e^{i(n-2k)\vp}\mu_{n-2k}( r )$.  \\

Choose $b = (\rho + \sqrt{\rho})^{-2k}\frac{\mu_n\pr{\rho + \sqrt{\rho}}}{\mu_{n-2k}\pr{\rho + \sqrt{\rho}}}$ so that $|u_1(\rho + \sqrt{\rho}, \vp)| = |u_2(\rho + \sqrt{\rho}, \vp)|$.  Since $|u_2(r, \vp)/u_1(r, \vp)| =\abs{ br^{2k}\frac{\mu_{n-2k}( r )}{\mu_n\pr{r}}}$, then by the assumptions on $k$ and $\rho$ and the behavior of $\mu_n$ and $\mu_{n-2k}$,
\begin{align}
&|u_2(r, \vp)/u_1(r, \vp)| \le e^{-C}, \quad r \in \brac{\rho, \rho + \tfrac{2}{3}\sqrt{\rho}}
\label{uBd1} \\
&|u_2(r, \vp)/u_1(r, \vp)| \ge  e^{C}, \quad r \in \brac{\rho + \tfrac{4}{3}\sqrt{\rho}, \rho + 2\sqrt{\rho}}.
\label{uBd2}
\end{align}
 
 Choose smooth monotonic cutoff functions $\psi_1$, $\psi_2$, $\psi_3$, $\psi_4$ such that 
 $$\psi_1(r) = \left\{ \begin{array}{rl} 
 1 & r \in\brac{\rho, \rho + \tfrac{1}{3}\sqrt{\rho}} \\
 \frac{1}{2} & r \in \brac{\rho + \tfrac{2}{3}\sqrt{\rho}, \rho + \tfrac{4}{3}\sqrt{\rho}} \\
 0 & r \in \brac{\rho + \tfrac{5}{3}\sqrt{\rho}, \rho + 2\sqrt{\rho}} 
 \end{array}\right., 
\psi_2(r) = \left\{ \begin{array}{rl} 
0 & r \in \brac{\rho, \rho + \tfrac{1}{3}\sqrt{\rho}} \\ 
\frac{1}{2} & r \in \brac{\rho + \tfrac{2}{3}\sqrt{\rho}, \rho + \tfrac{4}{3}\sqrt{\rho}}  \\
1 & r \in \brac{\rho + \tfrac{5}{3}\sqrt{\rho}, \rho + 2\sqrt{\rho}}
\end{array}\right.,$$
$$\psi_3\pr{r} = \left\{ \begin{array}{rl} 
1 & r \le \rho + \tfrac{5}{3}\sqrt{\rho} \\ 
0 & r \ge \rho + 1.9\sqrt{\rho} 
\end{array}\right.,
\psi_4\pr{r} = \left\{ \begin{array}{rl} 
0 & r \le \rho + 0.1\sqrt{\rho} \\ 
1 & r \ge \rho + \tfrac{1}{3}\sqrt{\rho} 
\end{array}\right..$$  
We require that 
\begin{equation}
\psi_1 + \psi_2 \le 1.
\label{sumBd}
\end{equation}  
This bound on the sum of the cutoff functions is the significant difference between the current construction and those that appeared in \cite{D13}.   Moreover, ensure that
\begin{equation}
0 \le |\psi_i(r)| \le 1 \quad\mathrm{and}\quad |\psi_i^{(j)}(r)| \le Cr^{-j/2} \quad \forall r \in \R^+, i= 1,2,3,4, j = 1, 2.  
\label{1psi}
\end{equation}  
Let
\begin{align*}
\phi_{a,b}\pr{r} &= -\frac{1}{2}\int \frac{\la r}{\sqrt{a^2 - \la r^2}\sqrt{b^2 - \la r^2}} dr = -\frac{1}{4}\log \pr{2\sqrt{a^2 - \la r^2}\sqrt{b^2 - \la r^2} + 2\la r^2 - a^2 - b^2}.
\end{align*}
When $a, b = n + \bigO\pr{k}$, $\phi \pr{r}= \bigO\pr{\log r}$, $\phi^\prime\pr{r} = \bigO\pr{r^{-1}}$, and $\phi^{\prime\prime}\pr{r} = \bigO\pr{r^{-2}}$.

Set 
$$u = \psi_1u_1\exp\pr{\psi_4 \phi_{n, n-2k}} + \psi_2u_2\exp\pr{\psi_3 \phi_{n, n-2k}}.$$
For the rest of step 1, we will abbreviate $\phi_{n,n-2k}$ with $\phi$. \\

\nid \emph{Step 1A: $r \in \brac{\rho, \rho + \tfrac{2}{3}\sqrt{\rho}} \cup \brac{\rho + \tfrac{4}{3}\sqrt{\rho}, \rho + 2\sqrt{\rho}}$}.

If $r \in \brac{\rho, \rho + \tfrac{2}{3}\sqrt{\rho}}$, then since $\psi_3 = \psi_4$ on the support of $\psi_2$, then on this annulus,
$$u = \psi_1 u_1\exp\pr{\psi_4 \phi} + \psi_2u_2\exp\pr{\psi_3 \phi} =  \brac{\psi_1 u_1 + \psi_2u_2}\exp\pr{\psi_4 \phi},$$
and by (\ref{uBd1}),
\begin{equation}
\exp\pr{-\psi_4 \phi}|u| \ge \abs{\psi_1}|u_1| - \abs{\psi_2}|u_2| \ge \frac{1}{2}\pr{\abs{u_1} -\abs{u_2}}\ge \frac{1}{2}\pr{1 - e^{-C}}|u_1| \ge e^{-C^\pri}|u_2| > 0.
\label{1AuN}
\end{equation}

If $r \in  \brac{\rho + \tfrac{4}{3}\sqrt{\rho}, \rho + 2\sqrt{\rho}}$, then since $\psi_3 = \psi_4$ on the support of $\psi_1$, then on this annulus,
$$u = \psi_1 u_1\exp\pr{\psi_4 \phi} + \psi_2u_2\exp\pr{\psi_3 \phi} = \brac{\psi_1 u_1 + \psi_2 u_2} \exp\pr{\psi_3 \phi},$$
and by (\ref{uBd2}),
\begin{equation}
\exp\pr{-\psi_3 \phi}|u| \ge \abs{\psi_2} |u_2| - \abs{\psi_1} |u_1| \ge \frac{1}{2}(1 - e^{-C})|u_2| \ge e^{-C^\pri}|u_1| > 0.
\label{1BuN}
\end{equation}

We see that
\begin{align*}
\LP u + \la u &= \pr{D_1  + \tilde D_{1,j}}\exp\pr{\psi_j \phi},
\end{align*}
where $j = 4$ if $r \in \brac{\rho, \rho + \tfrac{2}{3}\sqrt{\rho}}$ and $j = 3$ if $r \in \brac{\rho + \tfrac{4}{3}\sqrt{\rho}, \rho + 2\sqrt{\rho}}$.  Furthermore
\begin{align} 
&D_1 = \brac{ \psi_1\frac{\la}{\sqrt{n^2 - \la r^2}} + \psi_1^\pri \frac{1-2\sqrt{n^2 - \la r^2}}{r} + \psi_1^{\pri\pri} } u_1
\nonumber \\
&+ \brac{\psi_2\pr{\frac{\la}{\sqrt{(n-2k)^2 - \la r^2}} - \frac{8nk + 2(n+2k)\Phi^\pri + (\Phi^\pri)^2 - i\Phi^{\pri\pri}}{r^2} } + \psi_2^\pri\frac{1 - 2\sqrt{\pr{n-2k}^2 - \la r^2}}{r}+ \psi_2^{\pri\pri} }u_2 \nonumber \\
&\abs{D_1} \le C r^{-1/2} \pr{\abs{u_1} + \abs{u_2}}
\label{D_1ord} \\
&\tilde D_{1,j} = 2\brac{\pr{\psi_1^\prime -\psi_1 \tfrac{\sqrt{n^2 - \la r^2}}{r}}u_1 +\pr{\psi_2^\pri - \psi_2\tfrac{\sqrt{\pr{n-2k}^2 - \la r^2}}{r}}u_2 }\pr{\psi_j^\pri \phi + \psi_j \phi^\pri} 
\nonumber \\
&+ \brac{\frac{1}{r}\pr{\psi_j^\pri \phi + \psi_j \phi^\pri} + \pr{\psi_j^\pri \phi + \psi_j \phi^\pri}^2 + \pr{\psi_j^{\pri\pri} \phi + 2\psi_j^\pri \phi^\pri+ \psi_j \phi^{\pri\pri}}}\pr{\psi_1 u_1 + \psi_2 u_2 } \nonumber \\
&\abs{\tilde D_{1,j}} \le C\frac{\log r}{r^{1/2}} \pr{\abs{u_1} + \abs{u_2}}.
\label{tD_1ord} 
\end{align}
Let $\disp V = \frac{{D_1+\tilde D_{1,j}}}{u}\exp\pr{\psi_j \phi}$.  By (\ref{1AuN}) and (\ref{1BuN}), $\disp \abs{\frac{u_1 \exp\pr{\psi_j \phi}}{u}}$ and $\disp \abs{\frac{u_2 \exp\pr{\psi_j \phi}}{u}}$ are bounded.  Therefore, by (\ref{D_1ord}) and (\ref{tD_1ord}), (\ref{Vbd}) holds.  
This completes step 1A. \\ \\
\emph{Step 1B: $r \in \brac{\rho + \tfrac{2}{3}\sqrt{\rho}, \rho + \tfrac{4}{3}\sqrt{\rho}}$.}

On this annulus, $\psi_j \equiv \frac{1}{2}$ for $j = 1, 2$, and $\psi_j \equiv 1$ for $j = 3,4$, so 
$$u = \tfrac{1}{2}\pr{u_1 + u_2}\exp\pr{\phi} = \tfrac{1}{2}\pr{r^{-n}e^{-in\vp}\mu_n\pr{r} - br^{-n+2k}e^{iF(\vp)}\mu_{n-2k}\pr{r}}\exp\pr{\phi},$$
Since $u_2 = -br^{-n+2k}e^{i(n-2k)}\mu_{n-2k}( r )$ on $\set{|\vp -\vp_m| \le \frac{T}{5}}$ for $m = 0, 1, \ldots, 2n+2k-1$, then we will first consider these regions.  We have
\begin{align*}
\LP u + \la u &= J_1u,
\end{align*}
where 
\begin{align} 
J_1 &=  \frac{\la}{\sqrt{(n-2k)^2 - \la r^2}} +  \frac{\la}{\sqrt{n^2 - \la r^2}} + \frac{\phi^\pri}{r} + \pr{\phi^\pri}^2 + \phi^{\pri\pri} = \bigO\pr{r^{-1}}.
\label{J_1ordN} 
\end{align}
\\
Then $\abs{V} = \abs{J_1} \le C r^{-1}$.
\\
Now we will consider the annular sectors 
$$P_m = \set{(r, \vp) : r \in \brac{\rho + \tfrac{2}{3}\sqrt{\rho}, \rho + \tfrac{4}{3}\sqrt{\rho}}, \vp_m + \frac{T}{5} \le \vp \le \vp_m + \frac{4T}{5}}, \quad \textrm{for} \; m = 0, 1, \ldots, 2n+2k-1.$$
Notice that
\begin{equation}
|u| = \tfrac{1}{2}|br^{-n+2k}\mu_{n-2k}( r )\exp\pr{\phi}|\abs{e^{i(F(\vp) + n\vp)}- \frac{\mu_n}{br^{2k}\mu_{n-2k}}} 
= \tfrac{1}{2}|u_2|\abs{\exp\pr{\phi}}\abs{e^{i(F(\vp) + n\vp)}- \frac{ r^{-2k}\mu_n}{b\mu_{n-2k}}}.
\label{s4uBd}
\end{equation}
We study the behaviour of $S(\vp) = F(\vp) + n\vp$.  On the segment $[\vp_m, \vp_{m+1}]$, by (\ref{FDef}), $S(\vp) = (2n+2k)\vp + \Phi(\vp)$.  Thus $S(\vp_m) = 2\pi m$ and $S(\vp_{m+1}) = 2\pi(m+1)$.  Moreover,
$$S^\pri(\vp) = 2n +2k + \Phi^\pri(\vp) = 2n +2k + f(\vp).$$
By (\ref{f1}) and the conditions on $n$ and $k$, it may be assumed that $S^\pri(\vp) > n > 0$.  That is, $S$ increases monotonically on $[\vp_m, \vp_{m+1}]$. Therefore, if
$$\vp_m + T/5 \le \vp \le \vp_m + 4T/5,$$
then
$$2\pi m + \frac{nT}{5} \le S(\vp) \le 2\pi(m+1) - \frac{nT}{5},$$
or
$$2\pi m + \frac{n\pi}{5(n+k)} \le S(\vp) \le 2\pi(m+1) - \frac{n\pi}{5(n+k)}.$$
Since $k = \bigO(n^{1/2})$, then for $\disp \vp \in \brac{\vp_m + \frac{T}{5}, \vp_m + \frac{4T}{5}}$ we may assume that 
\begin{equation}
2\pi m + \frac{\pi}{7} \le S(\vp) \le 2\pi(m+1) - \frac{\pi}{7}.
\label{Sbd}
\end{equation}
From Lemma D.1 in \cite{D13} and (\ref{Sbd}), it follows that $\disp \abs{e^{iS(\vp)} -\frac{ r^{-2k}\mu_n}{b\mu_{n-2k}}} \ge \frac{1}{2}\sin\pr{\frac{\pi}{7}}$.  Therefore, by (\ref{s4uBd}),
\begin{equation}
|u(r,\vp)| \ge\tfrac{1}{4} |u_2(r, \vp)|\abs{\exp\pr{\phi}}\sin\pr{\frac{\pi}{7}}, \quad (r,\vp) \in P_m, \;\; m = 0, 1, \ldots, 2n +2k-1.
\label{1CuN}
\end{equation}

Then
\begin{align*}
\LP u + \la u &= J_1 u + \tfrac{1}{2} K_1 u_2 \exp\pr{\phi},
\end{align*}
where 
\begin{align} 
K_1 &= -\brac{\frac{8nk + 2(n+2k)\Phi^\pri + (\Phi^\pri)^2 - i\Phi^{\pri\pri}}{r^2}}  = \bigO\pr{r^{-1/2}}.
\label{K_1ord}
\end{align}

Let $\disp V = J_1 + \tfrac{1}{2} K_1\frac{u_2 \exp\pr{\phi}}{u}$.  It follows from (\ref{J_1ordN}), (\ref{1CuN}) and (\ref{K_1ord}) that $|V| \le Cr^{-1/2}$.
This completes step 1C. \\ 
\\
\emph{Step 2: $r \in \brac{\rho + 2\sqrt{\rho}, \rho + 3\sqrt{\rho}}$}. The solution $u_2 = -br^{-n+2k}e^{iF(\vp)}\mu_{n-2k}( r )$ is rearranged to a function of the form $u_3 = -br^{-n+2k}e^{i(n+2k)\vp}\mu_{n-2k}( r )$.

Choose a smooth cutoff function $\psi$ such that $\psi(r) = \left\{ \begin{array}{rl} 1 & r \le \rho + \tfrac{7}{3}\sqrt{\rho} \\ 0 & r \ge \rho + \tfrac{8}{3}\sqrt{\rho} \end{array} \right.$ and 
\begin{equation}
|\psi^{(j)}(r)| \le Cr^{-j/2} \quad j=0,1,2, \quad r \in\R^+.
\label{psiB}
\end{equation}
Set $u = -br^{-n+2k}\exp i\brac{\psi(r)\Phi(\vp) + (n+2k)\vp}\mu_{n-2k}( r )$.  Then
\begin{align*}
\LP u + \la u &= D_2 u,
\end{align*}
where
\begin{align}
D_2 &= -\tfrac{8nk}{r^2} - (\psi^\pri\Phi)^2 + i\Phi\pr{\frac{\psi^\pri}{r} + \psi^{\pri\pri}} -2\tfrac{(n+2k)}{r^2}\psi\Phi^\pri - \frac{(\psi\Phi^\pri)^2}{r^2} + i\frac{\psi\Phi^{\pri\pri}}{r^2} + \tfrac{\la}{\sqrt{(n-2k)^2 - \la r^2}} - 2i\psi^\pri \Phi\tfrac{\sqrt{(n-2k)^2 -\la r^2}}{r}
\nonumber \\
&= \bigO\pr{n \cdot k \cdot r^{-2}}.
\label{D_2ord}
\end{align}
Let $\disp V = D_2$ so that by (\ref{D_2ord}), $|V| \le Cr^{-1/2}$.  
This completes step 2. \\ 
\\
\emph{Step 3: $r \in \brac{\rho + 3\sqrt{\rho}, \rho + 4\sqrt{\rho}}$}.  The solution $u_3 = -br^{-n+2k}e^{i(n+2k)\vp}\mu_{n-2k}( r )$ is rearranged to $u_4 = -b_1r^{-n-2k}e^{i(n+2k)\vp}\mu_{n+2k}( r )$.

Choose a smooth cutoff function $\psi$ such that $\psi(r) = \left\{ \begin{array}{rl} 1 & r \le \rho + \tfrac{10}{3}\sqrt{\rho} \\ 0 & r \ge \rho + \tfrac{11}{3}\sqrt{\rho} \end{array} \right.$ and $\psi$ satisfies condition (\ref{psiB}).  Let $\disp d = (\rho + 3\sqrt{\rho})^{4k}\frac{\mu_{n-2k} (\rho + 3\sqrt{\rho})}{\mu_{n+2k}(\rho + 3\sqrt{\rho})}$, so that $g( r ) = dr^{-4k}\frac{\mu_{n+2k} ( r )}{\mu_{n-2k}( r )}$ satisfies $1 \ge \abs{g\pr{r}} \ge e^{-C}$ for all $r \in \brac{\rho + 3\sqrt{\rho}, \rho + 4\sqrt{\rho}}$.  Set $b_1 = bd$.  Let 
$$u = u_3\brac{\psi + (1-\psi)g} = \left\{ \begin{array}{rl} u_3 & r \le \rho + \tfrac{10}{3}\sqrt{\rho} \\ u_4 & r \ge \rho + \tfrac{11}{3}\sqrt{\rho} \end{array}\right..$$
Let $h(r)=\psi + (1-\psi)g$.  Since $\disp g^\pri( r ) = \brac{-\tfrac{4k}{r} - \tfrac{2\la k r}{(n+2k)(n-2k)} + \bigO\pr{\tfrac{kr^3}{n^4}}}g( r )$ and \\
$\disp g^{\pri\pri}( r ) = \brac{\tfrac{4k}{r^2} - \tfrac{2\la k}{(n+2k)(n-2k)} + \tfrac{16k^2}{r^2} + \tfrac{16\la k^2}{(n+2k)(n-2k)} + \bigO\pr{\tfrac{k^2r^2}{n^4}}}g( r )$, then for all $r \in \brac{\rho + 3\sqrt{\rho}, \rho + 4\sqrt{\rho}}$,
\begin{align}
\abs{h\pr{r}} &\ge e^{-\tilde{C}},
\label{hBd} \\
\abs{h^\pri\pr{r}} &\le C r^{-1/2},
\label{hpBd} \\
\abs{\LP h\pr{r}} &\le C r^{-1}.
\label{LhBd}
\end{align}
Then
\begin{align*}
\LP u + \la u &= D_3 u,
\end{align*}
where
\begin{align}
D_3 &= -\tfrac{8nk}{r^2} + \tfrac{\la}{\sqrt{(n-2k)^2 - \la r^2}} - 2\tfrac{\sqrt{(n-2k)^2 - \la r^2}}{r}\frac{h^\pri}{h} + \frac{\LP h}{h} = \bigO\pr{n \cdot k \cdot r^{-2} }.
\label{D_3ord}
\end{align}
Let $\disp V = D_3$ so that by (\ref{D_3ord}), $|V| \le Cr^{-1/2}$.
This completes step 3. \\ 
\\
\emph{Step 4: $r \in \brac{\rho + 4\sqrt{\rho}, \rho + 6\sqrt{\rho}}$}.  The solution $u_4 = -b_1r^{-n-2k}e^{i(n+2k)\vp}\mu_{n+2k}( r )$ is rearranged to $u_5 = ar^{-(n+k)}e^{-i(n+k)\vp}\mu_{n+k}( r )$.

Choose $\disp a = b_1(\rho + 5\sqrt{\rho})^{-k}\frac{\mu_{n+2k}(\rho + 5\sqrt{\rho})}{\mu_{n+k}(\rho + 5\sqrt{\rho})}$ to ensure that $|u_4(\rho + 5\sqrt{\rho}, \cdot)| = |u_5(\rho + 5\sqrt{\rho}, \cdot)|$.  Since $\disp \abs{u_5(r, \vp)/u_4(r, \vp)} = \abs{\pr{\frac{r}{\rho + 5\sqrt{\rho}}}^k\frac{\mu_{n+2k}(\rho + 5\sqrt{\rho})}{\mu_{n+k}(\rho + 5\sqrt{\rho})}\frac{\mu_{n+k}( r )}{\mu_{n+2k}( r )}}$, then by the assumptions on $k$ and $\rho$,
\begin{align}
|u_5(r, \vp)/u_4(r, \vp)| &\le e^{-C}, \quad r \in \brac{\rho + 4\sqrt{\rho}, \rho + \tfrac{14}{3}\sqrt{\rho}},
\label{uBd3} \\
|u_5(r, \vp)/u_4(r, \vp)| &\ge  e^{C}, \quad r \in \brac{\rho + \tfrac{16}{3}\sqrt{\rho}, \rho + 6\sqrt{\rho}}.
\label{uBd4}
\end{align}
Choose smooth cutoff functions $\psi_1$, $\psi_2$, $\psi_3$, $\psi_4$ such that 
$$\psi_1(r) = \left\{ \begin{array}{rl} 
1 &r \in \brac{\rho + 4\sqrt{\rho}, \rho + \tfrac{13}{3}\sqrt{\rho}} \\ 
\tfrac{1}{2} & r \in \brac{\rho + \tfrac{14}{3}\sqrt{\rho}, \rho + \tfrac{16}{3}\sqrt{\rho}} \\
0 & r \in \brac{ \rho + \tfrac{17}{3}\sqrt{\rho}, \rho + 6\sqrt{\rho}}
\end{array}\right., 
\psi_2(r) = \left\{ \begin{array}{rl} 
0 & r \in \brac{\rho + 4\sqrt{\rho}, \rho + \tfrac{13}{3}\sqrt{\rho}} \\ 
\tfrac{1}{2} & r \in \brac{\rho + \tfrac{14}{3}\sqrt{\rho}, \rho + \tfrac{16}{3}\sqrt{\rho}} \\
1 & r \in \brac{ \rho + \tfrac{17}{3}\sqrt{\rho}, \rho + 6\sqrt{\rho}}
\end{array}\right.,$$ 
$$\psi_3\pr{r} = \left\{ \begin{array}{rl} 
1 & r \le \rho + \tfrac{17}{3}\sqrt{\rho} \\ 
0 & r \ge \rho + 5.9\sqrt{\rho} 
\end{array}\right.,
\psi_4\pr{r} = \left\{ \begin{array}{rl} 
0 & r \le \rho + 4.1\sqrt{\rho} \\ 
1 & r \ge \rho + \tfrac{13}{3}\sqrt{\rho} 
\end{array}\right..$$  
Furthermore, we require that (\ref{sumBd}) holds and that each cutoff function satisfy condition (\ref{1psi}). 

We set 
$$u = \psi_1u_4\exp\pr{\psi_4 \phi_{n+k, n+2k}} + \psi_2u_5\exp\pr{\psi_3 \phi_{n+k, n+2k}}.$$
For the rest of step 4, we will abbreviate $\phi_{n+k,n+2k}$ with $\phi$. \\
\\
\emph{Step 4A: $r \in \brac{\rho + 4\sqrt{\rho}, \rho + \tfrac{14}{3}\sqrt{\rho}} \cup \brac{\rho + \tfrac{16}{3}\sqrt{\rho}, \rho + 6\sqrt{\rho}}$}. \\
If $r \in \brac{\rho + 4\sqrt{\rho}, \rho + \tfrac{14}{3}\sqrt{\rho}}$, then since $\psi_3 = \psi_4$ on the support of $\psi_2$, then on this annulus
$$u =\psi_1 u_4 \exp\pr{\psi_4 \phi} + \psi_2 u_5 \exp\pr{\psi_3 \phi} = \brac{\psi_1 u_4 + \psi_2 u_5}\exp\pr{\psi_4 \phi},$$
and by (\ref{uBd3}),
\begin{equation}
\exp\pr{-\psi_4 \phi} |u| \ge \abs{\psi_1}|u_4| - \abs{\psi_2}|u_5| \ge \tfrac{1}{2}(1 - e^{-C})|u_4| \ge e^{C^\pri}|u_5| > 0.
\label{uBd5}
\end{equation}

If $r \in \brac{\rho + \tfrac{16}{3}\sqrt{\rho}, \rho + 6\sqrt{\rho}}$, then since $\psi_3 = \psi_4$ on the support of $\psi_1$, we have
$$u = \psi_1 u_4\exp\pr{\psi_4 \phi} +\psi_2 u_5 \exp\pr{\psi_3 \phi} = \brac{\psi_1 u_4 + \psi_2 u_5}\exp\pr{\psi_3 \phi}.$$
so that by (\ref{uBd4}),
\begin{equation}
\exp\pr{-\psi_3 \phi} |u| \ge\abs{\psi_2} |u_5| - \abs{\psi_1} |u_4| \ge \tfrac{1}{2}(1 -e^{-C})|u_5| \ge e^{-C^\pri}|u_4| > 0.
\label{uBd6}
\end{equation}
We have
\begin{align*}
\LP u + \la u&= \pr{D_4 + \tilde D_{4,j}}\exp\pr{\psi_j \phi},
\end{align*}
where $j = 4$ if $r \in \brac{\rho + 4\sqrt{\rho}, \rho + \tfrac{14}{3}\sqrt{\rho}}$ and $j = 3$ is $r \in \brac{\rho + \tfrac{16}{3}\sqrt{\rho}, \rho + 6\sqrt{\rho}}$.  Furthermore,
\begin{align}
&D_4 = \brac{\psi_1\tfrac{\la}{\sqrt{(n+2k)^2 - \la r^2}} + \psi_1^\pri \tfrac{1-2\sqrt{(n+2k)^2 - \la r^2}}{r} + \psi_1^{\pri\pri} }u_4 
+ \brac{\psi_2\tfrac{ \la}{\sqrt{(n+k)^2 - \la r^2}} + \psi_2^\pri \tfrac{1-2\sqrt{(n+k)^2 - \la r^2}}{r} + \psi_2^{\pri\pri}}u_5 \nonumber \\
&\abs{D_4} \le C r^{-1/2} \pr{\abs{u_4} + \abs{u_5}}
\label{D_4ord} \\
&\tilde D_{4,j} = 2\brac{\pr{\psi_1^\pri -\psi_1\tfrac{\sqrt{(n+2k)^2 - \la r^2}}{r}}u_4 + \pr{\psi_2^\pri - \psi_2\tfrac{\sqrt{(n+k)^2 - \la r^2}}{r}} u_5 }\pr{\psi_j^\pri\phi + \psi_j \phi^\pri} 
\nonumber \\
&+ \brac{\frac{1}{r}\pr{\psi_j^\pri \phi + \psi_j \phi^\pri} + \pr{\psi_j^\pri \phi + \psi_j \phi^\pri}^2 + \pr{\psi_j^{\pri\pri} \phi + 2\psi_j^\pri \phi^\pri+ \psi_j \phi^{\pri\pri}}}\pr{\psi_1 u_4 + \psi_2 u_5} 
\nonumber \\
&\abs{\tilde D_{4,j} } \le C\frac{\log r}{r^{1/2}} \pr{\abs{u_4} + \abs{u_5}}.
\label{tD_4ord}
\end{align}
Let $\disp V = \frac{{D_4+\tilde D_{4,j}}}{u}\exp\pr{\psi_j \phi}$.  By (\ref{uBd5}) and (\ref{uBd6}), $\disp \abs{\frac{u_4 \exp\pr{\psi_j \phi}}{u}}$ and $\disp \abs{\frac{u_5 \exp\pr{\psi_j \phi}}{u}}$ are bounded.  Therefore, by (\ref{D_4ord}) and (\ref{tD_4ord}), (\ref{Vbd}) holds.  
This completes step 4A. \\ 
\\
\emph{Step 4B: $r \in \brac{\rho + \tfrac{14}{3}\sqrt{\rho}, \rho + \tfrac{16}{3}\sqrt{\rho}}$}.

On this annulus, since all cutoff functions are equivalent to $1$ or $\tfrac{1}{2}$, $u = \tfrac{1}{2}\pr{u_4 + u_5}\exp\pr{\phi}$ and
\begin{align*}
\LP u + \la u &= \tfrac{1}{2}J_4u,
\end{align*}
where 
\begin{align} 
J_4 &=  \frac{\la}{\sqrt{(n+2k)^2 - \la r^2}} +  \frac{\la}{\sqrt{(n+k)^2 - \la r^2}} + \frac{\phi^\pri}{r} + \pr{\phi^\pri}^2 + \phi^{\pri\pri} = \bigO\pr{r^{-1}} ,
\label{J_4ord} 
\end{align}
\\
Then with $V = \tfrac{1}{2}J_4$, $\abs{V} \le C r^{-1}$.
\\ 
We now prove the last statement of the lemma.  We first define a function $M(r)$ that will help estimate $m(r) = \max\set{|u(r,\vp)| : 0 \le \vp \le 2\pi}$.  Let
$$ M(r) = \left\{ \begin{array}{ll} 
r^{-n}\mu_{n}( r )\exp\pr{\psi_4 \phi_{n, n-2k}} & r \in \brac{\rho, \rho + \sqrt{\rho}} \\
br^{-n+2k}\mu_{n-2k}( r )\exp\pr{\psi_3 \phi_{n, n-2k}} & r \in \brac{\rho + \sqrt{\rho}, \rho + 2\sqrt{\rho}} \\
br^{-n+2k}\mu_{n-2k}( r ) & r \in \brac{\rho + 2\sqrt{\rho}, \rho + 3\sqrt{\rho}} \\
br^{-n+2k}\mu_{n-2k}( r )h( r ) & r \in \brac{\rho + 3\sqrt{\rho}, \rho + 4\sqrt{\rho}} \\
b_1r^{-(n+2k)}\mu_{n+2k}( r )\exp\pr{\psi_4 \phi_{n+k, n+2k}} & r \in \brac{\rho + 4\sqrt{\rho}, \rho + 5\sqrt{\rho} }\\
ar^{-(n+k)}\mu_{n+k}( r )\exp\pr{\psi_3 \phi_{n+k, n+2k}} & r \in \brac{\rho + 5\sqrt{\rho}, \rho + 6\sqrt{\rho}} \end{array}\right..$$
Note that $M(r)$ is equal to the modulus of the functions $u_1, \ldots, u_5$ from which our solution $u(r, \vp)$ is constructed.  Also, $M(r)$ is a continuous, piecewise smooth function for which $M(\rho) = m(\rho)$.  By the new condition (\ref{sumBd}) on the cutoff functions, we also have that $m(r) \le M(r)$.  It follows that
$$\ln m(r) - \ln m(\rho) \le \ln M(r) - \ln M(\rho).$$
Furthermore,
\begin{align*}
\frac{d}{dr} \ln \pr{r^{-n - ak}\mu_{n+ak}\pr{r}} &= - \frac{n+ak}{r} + \frac{n+ak}{r}\pr{1 - \sqrt{1 - \frac{\la r^2}{\pr{n+ak}^2}}}
\end{align*}
Since $\psi_j^\pri = \bigO\pr{r^{-1/2}}$, $\phi = \bigO\pr{\log r}$, $\disp \phi^\pri\pr{r} = \bigO\pr{r^{-1}}$, and $h^\pri(r) = \bigO\pr{r^{-1/2}}$, then everywhere on $\brac{\rho, \rho + 6\sqrt{\rho}}$, except at a finite number of points where $M(r)$ is not differentiable, we have
$$\frac{d}{dr}\ln M(r) = \frac{-n+\bigO\pr{k}}{r} \sqrt{1 - \frac{\la r^2}{\pr{n+\bigO\pr{k}}^2}} + \bigO\pr{r^{-1/2}\log r} \le -c,$$
by the conditions on $n$, $k$, $r$.  Therefore,
$$ \ln m(r) - \ln m(\rho) \le  \int_\rho^r(\ln M(t))^\pri dt \le - c\int_\rho^r dt,$$
proving the lemma.
\end{pf}

\section{Proof of Lemma \ref{meshP}}
\label{lemP}

We will now present the proof of Lemma \ref{meshP}, the slightly less-complicated construction.  Many of the steps in this proof are identical to those in the proof of Lemma \ref{meshN}, so we will often refer to them.

\begin{pf}[Proof of Lemma \ref{meshP}]
As $r$ increases from $\rho$ to $\rho + 6\sqrt{\rho}$, we rearrange equation (\ref{WPDE}) and its solution $u$ so that all of the above conditions are met.  This process is broken down into four major steps.

\nid\emph{Step 1: $r \in \brac{\rho, \rho + 2\sqrt{\rho}}$}.  During this step, the function $u_1 = r^{-n}e^{-in\vp}\mu_{n}( r )$ is rearranged to a solution of the form $u_2 = -br^{-n+2k}e^{iF(\vp)}\mu_{n-2k}( r )$, both of which satisfy an equation of the form (\ref{WPDE}), where $b$ and $F$ are as in the proof of Lemma \ref{meshN}.

Choose smooth cutoff functions $\psi_1$, $\psi_2$ as in step 1 of the proof of Lemma \ref{meshN}.
We set 
$$u = \psi_1u_1 + \psi_2u_2.$$

\nid \emph{Step 1A: $r \in \brac{\rho, \rho + \tfrac{2}{3}\sqrt{\rho}} \cup $}.
On this annulus, by (\ref{uBd1}),
\begin{equation}
|u| \ge \abs{\psi_1}|u_1| - \abs{\psi_2}|u_2| \ge \tfrac{1}{2}(1 - e^{-C})|u_1| \ge e^{-C^\pri}|u_2| > 0.
\label{1Au}
\end{equation}
Let $W = w(i\sin\vp, -i\cos\vp)$ for $w$ to be determined.  Then
\begin{align*}
W\cdot\gr u &= w d_1 \psi_1 u_1,\\
\LP u + \la u &= D_1,
\end{align*}
where 
\begin{align} 
d_1 &= -\frac{n}{r} + \frac{n + 2k + \Phi^\pri}{r}\frac{\psi_2 u_2}{\psi_1 u_1} = \bigO\pr{1}, 
\label{d_1ord}
\end{align}
and $D_1$ is as in (\ref{D_1ord}).
Since $\disp \left| \frac{u_2}{u_1} \right| \le e^{-C}$ then $\disp d_1 \ne 0$ and $\disp |d_1| \ge C$.  Let $\disp w = \frac{D_1}{d_1 \psi_1 u_1}$ so that by (\ref{1Au}), (\ref{d_1ord}) and (\ref{D_1ord}), $|W| \le Cr^{-1/2}.$ 
This completes step 1A. \\ 
\\
\emph{Step 1B: $r \in \brac{\rho + \tfrac{4}{3}\sqrt{\rho}, \rho + 2\sqrt{\rho}}$}.
On this annulus, by (\ref{uBd2}),
\begin{equation}
|u| \ge \abs{\psi_2}|u_2| - \abs{\psi_1}|u_1| \ge \tfrac{1}{2}(1 - e^{-C})|u_2| \ge e^{-C^\pri}|u_1| > 0.
\label{1Bu}
\end{equation}
Let $W = w(i\sin\vp, -i\cos\vp)$ for $w$ to be determined.  Then
\begin{align*}
W\cdot \gr u &= we_1\psi_2 u_2, \\
\LP u + \la u &= D_1,
\end{align*}
where 
\begin{align} 
e_1 &= -\frac{n}{r}\frac{\psi_1}{\psi_2}\frac{u_1}{u_2} + \frac{n + 2k + \Phi^\pri}{r} = \bigO\pr{1}, 
\label{e_1ord}
\end{align}
and $D_1$ is as in (\ref{D_1ord}).
Since $\disp \abs{\frac{u_1}{u_2}} \le e^{-C}$ then $\disp e_1 \ne 0$ and $\disp |e_1| \ge C$.  Let $\disp w = \frac{E_1}{e_1 \psi_2 u_2}$ so that by (\ref{1Bu}), (\ref{e_1ord}), and (\ref{D_1ord}), $|W| \le Cr^{-1/2}.$
This completes step 1B. \\ 
\\ 
\emph{Step 1C: $r \in \brac{\rho + \tfrac{2}{3}\sqrt{\rho}, \rho + \tfrac{4}{3}\sqrt{\rho}}$.}

On this annulus, $\psi_1 \equiv \tfrac{1}{2} \equiv \psi_2$, so 
$$u = \tfrac{1}{2}\pr{u_1 + u_2} = \tfrac{1}{2}\brac{r^{-n}e^{-in\vp}\mu_n\pr{r} - br^{-n+2k}e^{iF(\vp)}\mu_{n-2k}( r )}.$$
Let $W = w_1(\cos \vp, \sin \vp) + w_2(i\sin\vp, -i\cos\vp)$ for $w_1$, $w_2$ to be determined.  Then
\begin{align*}
W\cdot \gr u &= \tfrac{1}{2}\brac{w_1j_1(0)u_1 - w_2j_2(0)u_1 + w_1j_1(2)u_2  + w_2\pr{j_2(2) + \tilde j_2}u_2}, \\
\LP u + \la u &= \tfrac{1}{2}\brac{J_1(0)u_1 + \pr{J_1(2) + \tilde J_1} u_2},
\end{align*}
where 
\begin{align} 
j_1(a) &= -\tfrac{\sqrt{\pr{n-ak}^2 - \la r^2}}{r} = \bigO(1), 
\label{j_11ord} \\
j_2(a) &= \tfrac{n+ak}{r} = \bigO(1), 
\label{j_12ord} \\
\tilde j_2 &= \tfrac{\Phi^\pri}{r} = \bigO(r^{-1/2}), 
\label{tj_12ord} \\
J_1(a) &=  \tfrac{\la}{\sqrt{\pr{n-ak}^2 - \la r^2}} = \bigO\pr{r^{-1}} ,
\label{J_1ordP} \\
\tilde J_1 &= - \tfrac{8nk + 2(n+2k)\Phi^\pri + (\Phi^\pri)^2 - i\Phi^{\pri\pri}}{r^2} = \bigO\pr{r^{-1/2}}.
\label{tJ_1ord}
\end{align}
\\
 If we let
 \begin{align*}
 w_1 &= \frac{\pr{j_2(2) + \tilde j_2} J_1(0) + j_2(0) \pr{J_1(2) + \tilde J_1}}{j_1(0) \pr{j_2(2) + \tilde j_2} + j_2(0) j_1(2)} 
 = \bigO\pr{r^{- 1/2}} \\
 w_2 &= \frac{j_1(0) \pr{J_1(2) + \tilde J_1} - j_1(2) J_1(0)}{j_1(0) \pr{j_2(2) + \tilde j_2} + j_2(0) j_1(2)}\pr{J_1(2) + \tilde J_1}
 = \bigO\pr{r^{- 1/2}}, \\
 \end{align*}
then (\ref{WPDE}) is satisfied.   It follows that $|W| \le Cr^{-1/2}$.  This completes step 1C. \\ 
\\
\emph{Step 2: $r \in \brac{\rho + 2\sqrt{\rho}, \rho + 3\sqrt{\rho}}$}. The solution $u_2 = -br^{-n+2k}e^{iF(\vp)}\mu_{n-2k}( r )$ is rearranged to $u_3 = -br^{-n+2k}e^{i(n+2k)\vp}\mu_{n-2k}( r )$ by setting $u = -br^{-n+2k}\exp i\brac{\psi(r)\Phi(\vp) + (n+2k)\vp}\mu_{n-2k}( r )$, as in the proof of Lemma \ref{meshN}.  Let $W = w(\cos\vp, \sin\vp)$ for $w$ to be determined.  Then
\begin{align*}
W\cdot\gr u &= wd_2u, \\
\LP u + \la u &= D_2 u,
\end{align*}
where
\begin{align}
d_2 &=-\frac{\sqrt{(n-2k)^2 - \la r^2}}{r} + i\psi^\pri \Phi= \bigO\pr{1}
\label{d_2ord},
\end{align}
and $D_2$ is as in (\ref{D_2ord}).
By the conditions on $n$ and $k$, $d_2 \ne 0$ so $|d_2| \ge C$.  Let $\disp w = \frac{D_2}{d_2}$ so that by (\ref{d_2ord}) and (\ref{D_2ord}), $|W| \le Cr^{-1/2}.$  
This completes step 2. \\ 
\\
\emph{Step 3: $r \in \brac{\rho + 3\sqrt{\rho}, \rho + 4\sqrt{\rho}}$}.  The solution $u_3 = -br^{-n+2k}e^{i(n+2k)\vp}\mu_{n-2k}( r )$ is rearranged to $u_4 = -b_1r^{-n-2k}e^{i(n+2k)\vp}\mu_{n+2k}( r )$ by setting $u = u_3\brac{\psi + (1-\psi)g} = u_3 h$, as in Lemma \ref{meshN}.
Let $W = w(\cos\vp, \sin\vp)$ for some $w$ to be determined.  Then
\begin{align*}
W\cdot \gr u & = wd_3u, \\
\LP u + \la u &= D_3 u,
\end{align*}
where
\begin{align}
d_3 &= -\frac{\sqrt{(n-2k)^2 - \la r^2}}{r} + \frac{h^\pri}{h} = \bigO\pr{1},
\label{d_3ord}
\end{align}
and $D_3$ is as in (\ref{D_3ord}).  By (\ref{hpBd}) and the conditions on $n$ and $k$, $d_3 \ne 0$ so that $|d_3| \ge C$.  Let $\disp w = \frac{D_3}{d_3}$ so that by (\ref{d_3ord}) and (\ref{D_3ord}), $|W| \le Cr^{-1/2}.$ 
This completes step 3. \\ 
\\
\emph{Step 4: $r \in \brac{\rho + 4\sqrt{\rho}, \rho + 6\sqrt{\rho}}$}.  The solution $u_4 = -b_1r^{-n-2k}e^{i(n+2k)\vp}\mu_{n+2k}( r )$ is rearranged to $u_5 = ar^{-(n+k)}e^{-i(n+k)\vp}\mu_{n+k}( r )$.
Choose smooth cutoff functions $\psi_1$ and $\psi_2$ as in Lemma \ref{meshN} and set
$$u = \psi_1u_4 + \psi_2u_5.$$
\emph{Step 4A: $r \in \brac{\rho + 4\sqrt{\rho}, \rho + \tfrac{14}{3}\sqrt{\rho}}$}.

On this annulus, by (\ref{uBd3}),
\begin{equation}
|u| \ge \abs{\psi_1}|u_4| - \abs{\psi_2}|u_5| \ge \tfrac{1}{2}(1 - e^{-C})|u_4| \ge e^{-C^\pri}|u_5| > 0.
\label{uBd5P}
\end{equation}
Let $W = w(i\sin\vp, -i\cos\vp)$ for $w$ to be determined.  Since $\psi_1 \equiv 1$ in this annulus, we have
\begin{align*}
W\cdot \gr u &= wd_4 \psi_1 u_4, \\
\LP u + \la u&= D_4,
\end{align*}
where
\begin{align}
d_4 &=  \frac{n+2k}{r}  - \frac{n+k}{r}\frac{\psi_2}{\psi_1}\frac{u_5}{u_4} = \bigO\pr{1},
\label{d_4ord}
\end{align}
and $D_4$ is as in (\ref{D_4ord}).  Since $\disp \abs{\frac{u_5}{u_4}} \le e^{-C}$, then $|d_4| \ge C$.  Let $\disp w = \frac{D_4}{d_4  \psi_1 u_4}$ so that by (\ref{uBd5P}), (\ref{d_4ord}), and (\ref{D_4ord}), $|W| \le Cr^{-1/2}.$
This completes step 4A. \\ 
\\
\emph{Step 4B: $r \in \brac{\rho + \tfrac{16}{3}\sqrt{\rho}, \rho + 6\sqrt{\rho}}$}.

On this annulus, by (\ref{uBd4}),
\begin{equation}
|u| \ge \abs{\psi_2}|u_5| - \abs{\psi_1}|u_4| \ge \tfrac{1}{2}(1 -e^{-C})|u_5| \ge e^{-C^\pri}|u_4| > 0.
\label{uBd6P}
\end{equation}
Let $W = w(i\sin\vp, -i\cos\vp)$ for $w$ to be determined.  Since $\psi_2 \equiv 1$ in this annulus, we have
\begin{align*}
W\cdot \gr u &= we_4 \psi_2 u_5, \\
\LP u + \la u &= D_4 ,
\end{align*}
where
\begin{align}
e_4 &= \frac{n+2k}{r}\frac{\psi_1}{\psi_2}\frac{u_4}{u_5}  - \frac{n+k}{r} = \bigO\pr{1},
\label{e_4ord}
\end{align}
and $D_4$ is as in (\ref{D_4ord}).  Since $\disp \abs{\frac{u_4}{u_5}} < e^{-C}$, then $|e_4| \ge C$.  Let $\disp w = \frac{E_4}{e_4 \psi_2 u_5}$ so that by (\ref{uBd6P}), (\ref{e_4ord}) and (\ref{D_4ord}), $|W| \le Cr^{-1/2}.$
This completes step 4B. \\
\\ 
\emph{Step 4C: $r \in \brac{\rho + \tfrac{14}{3}\sqrt{\rho}, \rho + \tfrac{16}{3}\sqrt{\rho}}$}.

On this annulus, $u = \tfrac{1}{2}\pr{u_4 + u_5}$.
Let $W = w_1(\cos \vp, \sin \vp) + w_2(i\sin\vp, -i\cos\vp)$ for $w_1$, $w_2$ to be determined.  Then
\begin{align*}
W\cdot \gr u &= \tfrac{1}{2}\brac{w_1j_1(-2)u_4 + w_2j_2(2)u_4 + w_1j_1(-1)u_5  -w_2j_2(1)u_5}, \\
\LP u + \la u &= \tfrac{1}{2}\brac{J_1(-2)u_4 + J_1(-1) u_5},
\end{align*}
where $j_1$, $j_2$ and $J_1$ are given by (\ref{j_11ord}), (\ref{j_12ord}) and (\ref{J_1ordP}), respectively.
 If we let
 \begin{align*}
 w_1 &= \frac{j_2(1) J_1(-2) + j_2(2) J_1(-1)}{j_1(-2) j_2(1) + j_2(2) j_1(-1)} 
 = \bigO\pr{r^{-1}} \\
 w_2 &= \frac{ j_1(-1) J_1(-2) - j_1(-2) J_1(-1)}{ j_1(-2) j_2(1) + j_2(2) j_1(-1)} 
 = \bigO\pr{r^{-1}},
 \end{align*}
then (\ref{WPDE}) is satisfied with $\abs{W} \le C r^{-1/2}$.  This completes step 4C. \\
\\ 
The remainder of the proof is very similar to (but simpler) than that of Lemma \ref{meshN}.
\end{pf}

\section{The proof of Lemma \ref{meshNX}}
\label{lemNX}

The proof of Lemma \ref{meshNX} is nearly identical to that of Lemma \ref{meshN}; we simply need to show that the logarithmic factor may be removed from estimate (\ref{Vbd}) to get (\ref{VXbd}).  

\begin{proof}[Proof of Lemma \ref{meshNX}]
We must show that $\abs{\tilde D_{1,j}}, \le C r^{-1/2} \pr{\abs{u_1} + \abs{u_1}}$ and $\abs{\tilde D_{4,j}} \le C r^{-1/2} \pr{\abs{u_4} + \abs{u_5}}$.  Once this is established, since all other estimates for $\abs{V}$ satisfy (\ref{Vbd}), the proof of Lemma \ref{meshN} applies and the construction is complete.

Recall that
\begin{align*} 
\tilde D_{1,j} =& 2\brac{\pr{\psi_1^\prime -\psi_1 \tfrac{\sqrt{n^2 - \la r^2}}{r}}u_1 +\pr{\psi_2^\pri - \psi_2\tfrac{\sqrt{\pr{n-2k}^2 - \la r^2}}{r}}u_2 }\pr{\psi_j^\pri \phi + \psi_j \phi^\pri}  \\
&+ \brac{\frac{1}{r}\pr{\psi_j^\pri \phi + \psi_j \phi^\pri} + \pr{\psi_j^\pri \phi + \psi_j \phi^\pri}^2 + \pr{\psi_j^{\pri\pri} \phi + 2\psi_j^\pri \phi^\pri+ \psi_j \phi^{\pri\pri}}}\pr{\psi_1 u_1 + \psi_2 u_2 }  \\
\tilde D_{4,j} &= 2\brac{\pr{\psi_1^\pri -\psi_1\tfrac{\sqrt{(n+2k)^2 - \la r^2}}{r}}u_4 + \pr{\psi_2^\pri - \psi_2\tfrac{\sqrt{(n+k)^2 - \la r^2}}{r}} u_5 }\pr{\psi_j^\pri\phi + \psi_j \phi^\pri}  \\
&+ \brac{\frac{1}{r}\pr{\psi_j^\pri \phi + \psi_j \phi^\pri} + \pr{\psi_j^\pri \phi + \psi_j \phi^\pri}^2 + \pr{\psi_j^{\pri\pri} \phi + 2\psi_j^\pri \phi^\pri+ \psi_j \phi^{\pri\pri}}}\pr{\psi_1 u_4 + \psi_2 u_5}
\end{align*}
Since $n = \big\lfloor \sqrt{\la} \pr{\rho + 8 \sqrt{\rho}} \big\rfloor$, $k \sim 12\sqrt{\la}\sqrt{\rho}$, and $r \in \brac{\rho, \rho + 6 \sqrt{\rho}}$, then $\disp \frac{\sqrt{\pr{n+ak}^2 - \la r^2}}{r} = \bigO\pr{r^{-1/4}}$, and the result follows.
\end{proof}

\nid \textbf{Acknowledgement} The author is extremely grateful to her advisor, Carlos Kenig, for his  support and encouragement.

\bibliography{refs}

\begin{thebibliography}{}

\bibitem[\protect\citeauthoryear{Ching-Lung~Lin}{Ching-Lung~Lin}{2013}]{LW13}
Ching-Lung~Lin, J.-N.~W. (2013).
\newblock Quantitative uniqueness estimates for the general second order
  elliptic equations.
\newblock Preprint, arXiv:1303.2189.

\bibitem[\protect\citeauthoryear{Davey}{Davey}{}]{D13}
Davey, B.
\newblock Some quantitative unique continuation results for eigenfunctions of
  the magnetic {S}chr\"odinger operator.
\newblock {\em Communications in Partial Differential Equations\/}~{\em (to
  appear)}.

\bibitem[\protect\citeauthoryear{Meshkov}{Meshkov}{1989}]{M89}
Meshkov, V.~Z. (1989).
\newblock Weighted differential inequalities and their application for
  estimates of the decrease at infinity of the solutions of second-order
  elliptic equations.
\newblock {\em Trudy Mat. Inst. Steklov.\/}~{\em 190}, 139--158.
\newblock Translated in Proc. Steklov Inst. Math. {{\bf{1}}992}, no. 1,
  145--166, Theory of functions (Russian) (Amberd, 1987).

\end{thebibliography}
\bibliographystyle{chicago}

\end{document}